\documentclass[12pt]{amsart}
\usepackage{amsthm,amsmath,a4,amssymb,verbatim,url,lscape}

\newtheorem{thm}{Theorem}[section]
\newtheorem{cor}[thm]{Corollary}
\newtheorem{lem}[thm]{Lemma}
\newtheorem{prop}[thm]{Proposition}
\theoremstyle{definition}
\newtheorem{defn}[thm]{Definition}
\theoremstyle{remark}

\newtheorem{war}[thm]{Warning}

\numberwithin{equation}{section}
\newcommand{\PC}{\mathrm{PC}}
\newcommand{\proj}{\mathbf{Proj}}
\newcommand{\aff}{\mathbf{Aff}}
\newcommand{\Pone}{\mathbf{P}^1}
\newcommand{\PGL}{\textnormal{PGL}}
\newcommand{\PSL}{\textnormal{PSL}}
\newcommand{\SL}{\textnormal{SL}}

\newcommand{\Aut}{\textnormal{Aut}}

\newcommand{\Z}{\mathbf{Z}}
\newcommand{\R}{\mathbf{R}}

\begin{document}

\address{CNRS and Univ Lyon, Univ Claude Bernard Lyon 1, Institut Camille Jordan, 43 blvd. du 11 novembre 1918, F-69622 Villeurbanne}
\email{cornulier@math.univ-lyon1.fr}
\thanks{Supported by project ANR-16-CE40-0005 CoCoGro}




\title{Property~FW and 1-dimensional piecewise groups}
\author{Yves Cornulier}%
\date{November 22, 2020}


\begin{abstract}
Property~FW is a natural combinatorial weakening of Kazhdan's Property~T.
We prove that the group of piecewise homographic self-transfor\-mations of the real projective line, has ``few" infinite subgroups with Property~FW. In particular, 
no such subgroup is amenable or has Kazhdan's Property~T. These results are extracted from a longer paper. We provide a complete proof, whose main tools are the use of the notion of partial action and of one-dimensional geometric structures.
\end{abstract}
\maketitle

\section{Introduction}

The purpose of this note is to extract from the long paper \cite{Cmain} a strong restriction on groups of piecewise homographic self-transformations of the real projective line. We start with introducing this group, and then the rigidity properties we deal with, namely Kazhdan's Property~T and Property~FW. 

\subsection{Piecewise groups}

Let $\PC(\Pone)$ be the group of piecewise continuous ``self-transform\-ations" of the real projective line $\Pone$. Formally speaking, it consists of maps $f:\Pone\to\Pone$, continuous outside a finite subset, such that there exists another such map $g:\Pone\to\Pone$ such that $f\circ g$ and $g\circ f$ coincide with the identity map outside a finite subset. This identification essentially means that we ignore the values at discontinuity points.

We will mainly be interested in its subgroup $\PC_\proj(\Pone)$ consisting of piecewise homographic self-transformations. That is, it consists of those elements represented by a map $f:\Pone\to\Pone$ such that for all but finitely many $x\in\Pone$, there exists a neighborhood $V$ of $x$ and $g\in\PGL_2(\R)$ such that $f(x')=g(x')$ for all $x'\in V$.

It naturally contains various subgroups that have been considered by various people with various points of view:
\begin{itemize}
\item the group of piecewise affine self-transformations of $[0,1]$, and its subgroup of continuous elements. In turn, those piecewise affine groups contain many famous groups, notably Thompson's groups and several of their generalizations, and also the group $\mathrm{IET}$ of interval exchange transformation. The list of references is too long to fit here!
\item Define $G^i$ as the intersection $\PC_\proj(\Pone)\cap\mathrm{Diff}^i(\Pone)$. First observe that for $i\ge 2$, $G^i$ is reduced to $\PGL_2(\R)$. The groups $G^1\subset G^0=\PC(\Pone)\cap\mathrm{Homeo}(\Pone)$ are much larger.
The group $G^1$ was introduced by Strambach and Betten \cite{Str,Bet}; they identified it as automorphism group of a Moulton plane (some affine plane, in the sense of incidence geometry). It was further studied in \cite{BW,Gre}.
\item The intersection $G^0_\infty=\PC_\proj(\Pone)\cap\mathrm{Homeo}(\R)$ (stabilizer of $\infty\in\Pone$ in $G^0$ above) was considered by Monod \cite{Mon}, who remarkably observed that this group (and hence some of its finitely generated subgroups) is not amenable while it does not contain any free subgroup. The key argument for non-amenability is a topological notion of amenability for equivalence relations, due to Carri\`ere and Ghys \cite{CG}. Subsequently non-amenable finitely presented subgroups of $G^0_\infty$ were exhibited by Lodha and Moore~\cite{LM}.
\end{itemize}

A general question is to find restrictions about the possible structure of subgroups of $\PC_\proj(\Pone)$. Such questions have often been addressed within the various subgroups mentioned above (see, among others, the references in \cite{Cmain}). 

\subsection{Rigidity properties: Kazhdan and FW}
We will mostly be concerned with two group properties: Kazhdan's Property~T, a rigidity property about unitary representations, and one of its less known combinatorial weakenings, Property~FW. We first quickly overview these notions, before turning to the results. 

Kazhdan's Property~T was originally introduced by Kazhdan. For a countable discrete group $\Gamma$, let us use as a definition the following characterization due to Robertson and Steger \cite{RoS} ($\triangle$ denoting symmetric difference):

The countable discrete group $\Gamma$ has Property~T if\footnote{To respect established terminology, we could remove the word ``countable" here at the price of replacing ``Property~T" with ``Property FH". Nevertheless the case of $\Gamma$ uncountable is of marginal interest here.} for every measure space $(E,\mathcal{A},\mu)$ with a measure preserving $\Gamma$-action and measurable subset $X\subset E$ such that $\mu(X\triangle \gamma X)<\infty$ for every $\gamma\in\Gamma$, there exists a $\Gamma$-invariant measurable subset $X'\subset E$ such that $\mu(X\triangle X')<\infty$.

This means that the only way to satisfy the ``invariance modulo finite measure" condition is the trivial one: modifying an invariant subset on a set of finite measure. Here are two illustrating examples of actions in which such $X'$ does not exist:
\begin{itemize}
\item $\Z$ or $\R$ naturally acting on $E=\R$ by translation, $X=\R_{\ge 0}$;
\item $\PSL_2(\R)$ acting diagonally on $E'=\mathbb{P}^1\times\mathbb{P}^1$; this action is transitive on the complement $E$ of the diagonal, and preserves a Radon measure (unique up to rescaling). The space $E$ is homeomorphic to an open cylinder in which both ends have infinite measure, and $X$ can be chosen as one ``half" of it.
\end{itemize}

A natural combinatorial weakening of Property~T is when the previous condition is asked while $E$ is required to be a set endowed with the $\sigma$-algebra of all subsets and the counting measure. That is, the discrete group $\Gamma$ has Property~FW if for every $\Gamma$-set $E$, every subset $X\subset E$ such that $X\triangle \gamma X$ is finite for every $\gamma\in\Gamma$, coincides with a $\Gamma$-invariant subset modulo a finite subset. 

That $\Gamma$ has Property~FW has various geometric characterizations \cite{CFW}, among which
\begin{itemize}
\item every $\Gamma$-action on every nonempty CAT(0) cube complex by automorphisms, has a fixed point;
\item $\Gamma$ is finitely generated, and every Schreier graph $\Gamma/\Lambda$ has at most one end.
\end{itemize}

For a long time the class of countable discrete groups known to satisfy Property~T was essentially reduced to arithmetic lattices and analogues, such as $\SL_n(\Z)$, $\SL_n(\Z)\ltimes\Z^n$, or $\SL_n(\Z[1/2])$ for $n\ge 3$. Then this was considerably enlarged, in several directions. It was notably obtained, for many hyperbolic groups \cite{Zuk}, for analogues of arithmetic groups over general rings \cite{EJZ}, for automorphisms groups of free groups on $\ge 5$ generators \cite{KNO,KKN}. This list is far from comprehensive; the book \cite{BHV} surveys several of the developments until 2005.

The study of Property~FW for its own interest is much more recent \cite{CFW}. At this time, the main supply of groups with Property~FW consists of groups with Property~T. Examples of groups known to have Property~FW but not Property~T come from \cite{Cirr}; these are arithmetic lattices in products:
\begin{itemize}
\item Let $G$ be a semisimple Lie group with at least one noncompact simple factor with Property~T. Then every irreducible lattice in $G$ has Property~FW. On the other hand it does not have Property~T if $G$ does not, e.g., when $G=\mathrm{SO}(3,2)\times\mathrm{SO}(4,1)$.
\item For $G$ a semisimple Lie group of real rank $\ge 2$, conjecturally Property~FW still holds for all its irreducible lattices. The remaining case is when no noncompact simple factor has Property~T. In this case, it is nevertheless known in some significant examples. For instance, for $k\ge 2$ non-square, $\SL_2(\Z[\sqrt{k}])$ (which sits as an irreducible lattice in $\SL_2(\R)^2$) has Property~FW.
\end{itemize}
 Other examples mechanically follow: for instance using the Ollivier-Wise ``Rips machine with Kazhdan kernel" \cite{OW}, one deduces the existence of Gromov-hyperbolic groups with Property~FW but not Property~T.

The main reason for introducing using Property~FW (instead of sticking to its better known and more elaborate Property~T cousin) is therefore not, at this state of knowledge, the gain of generality. It is rather that this is exactly the assumption that is needed to run the argument, as we hope to convince the reader in the sequel.

\subsection{Results}
The purpose of this note is to prove the following theorem extracted from \cite{Cmain}.

\begin{thm}\label{proj_fewfw}
Let $\Gamma$ be an infinite subgroup of $\PC_\proj(\Pone)$ with Property~FW. Then there exist $n\ge 1$ and subgroups $W\le\Lambda\le\Gamma$, with $W$ finite normal, $\Lambda$ normal of finite index, such that $\Lambda/W$ can be embedded into $\PSL_2(\R)^n$ in a such a way that each projection $\Lambda\to\PSL_2(\R)$ has a Zariski-dense image.
\end{thm}

From the Tits alternative (in the particular easy case of subgroups of $\PSL_2(\R)$), we deduce:

\begin{cor}
Every infinite subgroup $\PC_\proj(\Pone)$ with Property~FW has a non-abelian free subgroup. In particular, $\PC_\proj(\Pone)$ has no infinite amenable subgroup with Property~FW.\qed
\end{cor}

Also, it is well-known that $\PSL_2(\R)$ has no infinite subgroup with Property~T. Indeed, since the locally compact group $\PSL_2(\R)$ has the Haagerup Property (Faraut-Harzallah \cite[Cor.\ 7.4]{FH}, reproved by Robertson \cite{Rob}), such a subgroup would be contained in a compact subgroup, and thus be abelian, and hence finite.

\begin{cor}
The group $\PC_\proj(\Pone)$ has no infinite subgroup with Kazhdan's Property~T.\qed
\end{cor}

(The same conclusion fails for Property~FW, since $\SL_2(\Z[\sqrt{2}])$ has Property~FW as mentioned above.) The first known result of this flavor (in the context of 1-dimensional dynamics) is maybe Navas' result \cite{Nav} that the group of diffeomorphisms of class $>3/2$ of the circle has no infinite subgroup with Kazhdan's property~T.

Our approach also addresses the group $\PC_\aff(\R/\Z)$ of piecewise affine self-transformations (which can be viewed as subgroup of the previous one, since $\PC_\proj(\Pone)$ is isomorphic to $\PC_\proj(\R/\Z)$, through a piecewise homographic transformation between $\R/\Z$ and $\Pone$, or alternatively by extending as the identity outside $[0,1]$). In this case, we have a stronger conclusion:

\begin{thm}\label{affine_nofw}
The group $\PC_\aff(\R/\Z)$ has no infinite subgroup with Property~FW.
\end{thm}

The same statement for its subgroup $\PC_\aff^0(\R/\Z)$ of continuous elements was independently proved by Lodha, Matte Bon, and Triestino \cite{LMT}. Even the case of Property~T is new in Theorem \ref{affine_nofw}; nevertheless the absence of infinite Property~T subgroups in its subgroup $\mathrm{IET}$ of piecewise translations was initially proved in \cite[Theorem 6.1]{DFG} with another approach.

The formalism of partial actions is very useful in the output proof. A regularization theorem in the context of birational actions of groups with Property~FW was recently obtained by the author \cite{C_bir} using similar concepts, but with a more involved proof. 

\smallskip
\noindent {\bf Acknowledgment.} I thank Pierre de la Harpe, Octave Lacourte and the referee for a number of corrections.

\section{Main concepts and auxiliary proofs}

Attempts to define partial actions were done many times, for instance in the context of integrating Lie algebras of vector fields by Palais \cite[Chap.\ III]{Pal}. Eventually a very general and flexible notion was introduced by Exel~\cite{E}.

\begin{defn}
A topological partial action of a (discrete) group $\Gamma$ on a topological space $X$ is an assignment $g\mapsto\alpha(g)$, where $\alpha(g)$ is a homeomorphism between two open subsets of $X$, satisfying the following conditions:
\begin{enumerate}
\item $\alpha(1_\Gamma)=\mathrm{id}_X$
\item $\alpha(g^{-1})=\alpha(g)^{-1}$, for all $g\in \Gamma$;
\item $\alpha(gh)$ extends $\alpha(g)\alpha(h)$, for all $g,h\in \Gamma$. 
\end{enumerate}
A partial action is called cofinite if for every $g\in\Gamma$, the domain of definition of $\alpha(g)$ is cofinite (= has finite complement) in $X$.
\end{defn}
Here $\alpha(g)^{-1}$ denotes the partial inverse, and $\alpha(g)\alpha(h)$ is  the composition: its graph consists of those $(x,x'')$ for which there exists $x'\in X$ such that $(x,x')$ belongs to the graph of $\alpha(h)$ and $(x',x'')$ belongs to the graph of $\alpha(g)$.

\begin{defn}
A globalization of a partial action $\alpha$ as above, is a continuous action $\beta$ of $\Gamma$ on a topological space $Y$, a homeomorphism $i$ from $X$ onto an open subset of $Y$, such that for all $g\in\Gamma$ and $x,x'\in X$, the element $(x,x')$ belongs to the graph of $\alpha(g)$ if and only if $(i(x),i(x'))$ belongs to the graph of $\beta(g)$. It is called essential if every $\Gamma$-orbit meets $i(X)$\end{defn}

In other words, viewing $i$ as an inclusion, this means that the partial action is obtained by restricting the action to the given open subset.

The following proposition is already present (with special hypotheses but the same straightforward proof) in \cite{Pal}, and asserted in full generality in \cite{AT,A,KL}.

\begin{prop}
Every partial action admits an essential globalization, unique up to unique isomorphism, called universal globalization. Moreover, the underlying set of the universal globalization coincides with the universal globalization of the partial action on the underlying (discrete) set; in other words, forgetting the topology of $X$ and taking the universal globalization commute.\qed
\end{prop}

Property~FW will be used in the following form (as in \cite{C_bir}):

\begin{defn}\label{d_fw}
A group $\Gamma$ has Property~FW if for every set $X$ and for every cofinite partial $\Gamma$-action on $X$, there exists a $\Gamma$-invariant subset $Y$ of its universal globalization, such that the symmetric difference $X\triangle Y$ is finite.
\end{defn}

For our purposes, let us just slightly strengthen the conclusion.

\begin{prop}\label{p_neu}
Let $\Gamma$ be a group with Property~FW. In the setting of Definition \ref{d_fw}, the $\Gamma$-invariant subset $Y$ can be chosen to satisfy: for every finite subset $F$ of $Y$, there exists $g\in\Gamma$ such that $gF\subset X$.
\end{prop}
\begin{proof}
First, choose $Y_0$ as in the definition, and define $Y$ as the complement in $Y_0$ of the union of finite $\Gamma$-orbits in $Y_0$ meeting the finite subset $Y_0\smallsetminus X$. Then for every $F$, a lemma due to B.H.\ Neumann \cite{Neu} (see also \cite[Lemma 6.25]{Poi}) ensures the existence of~$g$.
\end{proof}

An immediate but crucial observation is that if the partial action preserves some geometric structure, then this geometric structure is inherited by the universal globalization, and preserved by its $\Gamma$-action. Here for the sake of brevity, we only consider the following geometric structures: 
\begin{itemize}
\item 1-manifolds with an affine structure: charts valued in $\R$, with affine change of charts;
\item 1-manifolds with a projective structure: charts valued in $\Pone$, with homographic ($x\mapsto \frac{ax+b}{cx+d}$) change of charts.
\end{itemize}
We call them affinely-modeled and projectively-modeled 1-manifolds. Of course every affine structure defines a projective structure, and every projective structure defines a smooth (analytic) structure.

\begin{war}
Although we are mostly interested in Hausdorff manifolds, we do {\it not} assume that the manifolds are Hausdorff: a 1-manifold here is just a topological space locally homeomorphic to $\R$. The reason is that taking the universal globalization usually does not preserve being Hausdorff, and the proof needs to transit through this outlandish world. (Nevertheless, being a manifold implies the $T_1$-separation axiom: finite subsets are closed.)
\end{war}

The classification of connected Hausdorff affinely-modeled and projectively-modeled 1-manifolds was done by Kuiper \cite{Ku1,Ku2}, up to a minor (but subtle) error in \cite{Ku2} (see the appendix in~\cite{Cmain}).

While these notions are standard, we need to introduce this one:

\begin{defn}An affinely-modeled or projectively modeled 1-manifold is finitely-charted if it has a finite covering by bounded charts: here bounded means valued in a bounded interval of $\R$.
\end{defn}

Clearly this implies having finitely many components. Every compact affinely/ projectively-modeled 1-manifold is finitely-charted. But the affinely-modeled 1-manifolds $\R$ and $\R_{>0}$ are finitely-charted as projectively-modeled 1-manifold, but not as affinely-modeled 1-manifolds. An example of a connected Hausdorff non-finitely-charted projectively-modeled 1-manifold is the universal covering of $\Pone$.

A piecewise affine/homographic transformation between affinely/projectively-modeled 1-manifolds $X,Y$ is a locally affine/homographic isomorphism between cofinite subsets (identifying two such isomorphisms whenever they coincide on a cofinite subset). Denote by $\PC_\aff(X)$ and $\PC_\proj(X)$ the group of piecewise affine/homgographic self-transformations of $X$ (when it makes sense). A piecewise affine, resp.\ piecewise homographic, transformation $X\dasharrow Y$ induces an isomorphism $\PC_\aff(X)\to\PC_\aff(Y)$, resp.\ $\PC_\proj(X)\to\PC_\proj(Y)$.

\section{Regularization theorem and proofs}

\begin{thm}\label{main}
Let $\Gamma$ be a group with Property~FW. Let $X$ be a Hausdorff finitely-charted affinely-modeled [respectively projectively-modeled] 1-manifold. Let $\Gamma\to\PC_\aff(X)$ [resp.\ $\Gamma\to\PC_\proj(X)$] be a homomorphism. Then there exists a Hausdorff finitely-charted affinely-modeled (resp.\ projectively-modeled) 1-manifold $Y$, a piecewise affine [resp.\ piecewise homographic] transformation $X\dasharrow Y$, such that the induced map $\Gamma\to\PC_\aff(Y)$ [resp.\ $\Gamma\to\PC_\proj(Y)$] actually maps into the group of affine (resp.\ homographic) automorphisms of $Y$. 
\end{thm}
\begin{proof}
The two proofs are strictly similar (and applicable to other geometric structures), so let us do the affine case; the projective case just consists in changing the adequate word at each place denoted ($\ast$) below.

Define a partial action of $\PC_\aff(X)$ ($\ast$) on $X$, saying that $f$ is defined at $x$ if some representative $\bar{f}$ of $f$ is continuous and affine ($\ast$) at $x$: then $\bar{f}(x)$ does not depend on the choice of $\bar{f}$ (this just uses that $X$ has no isolated point). Let $X\to\hat{X}$ be the universal globalization. Using that $\Gamma$ has Property~FW, let $Y\subseteq\hat{X}$ be given by Proposition \ref{p_neu}. As a finite union of translates of an open subset of $X$ (namely $X\cap Y$), the subset $Y$ is open, and is a finitely-charted affinely-modeled ($\ast$) manifold. Since by the conclusion of Proposition \ref{p_neu}, any pair in $Y$ can be translated into $X$, and since $X$ is Hausdorff, we deduce that $Y$ is Hausdorff too.
\end{proof}

\begin{proof}[Proof of Theorem \ref{affine_nofw}] It follows as a corollary of Theorem \ref{main}: indeed, by Kuiper \cite{Ku1}, the connected Hausdorff finitely-charted affinely-modeled 1-manifolds are, up to isomorphism: the open interval $\mathopen] 0,1\mathclose[$, the standard circle $\R/\Z$, and the non-standard circles $\R_{>0}/\langle t\rangle$, where the latter means the quotient by the discrete subgroup generated by multiplication by $t$, where $t>1$ is a fixed number. For each such affine manifold, the affine automorphism group has an abelian subgroup of index $\le 2$. In particular, for an arbitrary Hausdorff finitely-charted affinely-modeled 1-manifold, the affine automorphism group is virtually abelian (indeed some finite index subgroup preserves each connected component). Since a virtually abelian group with Property~FW is finite, the conclusion follows.\end{proof}

Let us now deduce Theorem \ref{proj_fewfw} from Theorem \ref{main}. We need the following result which follows from classification (as almost achieved in \cite{Ku2}); however we give a short classification-free proof.

\begin{lem}\label{large_pro}
Let $X$ be a Hausdorff connected finitely-charted projectively-modeled 1-manifold, whose homographic automorphism group $\Aut_\proj(X)$ is not virtually metabelian. Then $X$ is isomorphic to a finite covering of $\Pone$.
\end{lem}
\begin{proof}
First suppose that $X$ is homeomorphic to an interval; then $X$ is isomorphic to some (necessarily non-empty) open interval $I$ in the universal covering $\widetilde{\Pone}$, whose oriented homographic automorphism group $\Aut^+_\proj(\widetilde{\Pone})$ can be identified to $\widetilde{\SL_2(\R)}$. Then $\Aut^+_\proj(I)$ is the stabilizer of $I$. Since point stabilizers for the action of $\widetilde{\SL_2(\R)}$ on $\widetilde{\Pone}$ are metabelian, we deduce that $\Aut^+_\proj(X)$ is metabelian, unless $X=\widetilde{\Pone}$, but the latter is excluded since $X$ is finitely-charted. 

Otherwise, $X$ is homeomorphic to the circle, and hence isomorphic to the quotient of some nonempty open interval $I$ in $\widetilde{\Pone}$ by a cyclic subgroup $\langle t\rangle$ of $\widetilde{\SL_2(\R)}$ acting freely and properly on $I$.  The oriented automorphism group of $X$ is therefore isomorphic to $N/\langle t\rangle$, where $N$ is the normalizer of $\langle t\rangle$ in $\widetilde{\SL_2(\R)}$. Then $N$ non-metabelian forces $N$ to be the whole group, which means that $t$ is central. This precisely means the desired conclusion.
\end{proof}

\begin{proof}[Proof of Theorem \ref{proj_fewfw}]Let $\Gamma$ be as in Theorem \ref{proj_fewfw}. By Theorem \ref{main}, we can suppose that $\Gamma$ is a subgroup of the automorphism group of a Hausdorff, finitely-charted projectively-modeled 1-manifold $Y$. Let $\Gamma^0$ be its normal subgroup of finite index consisting of elements preserving each component of $Y$ as well as its orientation: it also has Property~FW. Let $X$ be the union of components $Z$ of $Y$ such that the image of $\Gamma^0\to\mathrm{Homeo}(Z)$ is infinite. Then $X$ is $\Gamma$-invariant and the action of $\Gamma$ on $X$ is faithful on some finite index subgroup, so has finite kernel.  For every component $Z$ of $X$, the image of $\Gamma^0\to\mathrm{Homeo}(Z)$ is infinite with Property~FW, hence not virtually metabelian; hence by Lemma \ref{large_pro}, $Z$ is an $n$-fold covering of $\Pone$ for some $n\ge 1$; hence its oriented homographic automorphism isomorphic to $\PSL_2(\R)^{(k)}$, the connected $k$-fold covering of $\PSL_2(\R)$. Modding out by the center for each of the $n$ components of $X$, we obtain a homomorphism $\Gamma'\to\PSL_2(\R)^n$ with finite kernel, such that each projection has non-virtually-metabelian image, hence is Zariski-dense. Since $\Gamma$ is infinite, $n\ge 1$.\end{proof}

Let $\PC(\R/\Z)$ be the whole group of piecewise continuous self-transformations of $\R/\Z$. Whether it has an infinite Property~T subgroup is unknown, and precisely equivalent to a well-known open question.

\begin{prop}The following (absolute) statement are equivalent:
\begin{enumerate}
\item\label{t_hor} there is an infinite Property~T subgroup in $\mathrm{Homeo}(\R)$ (equivalently: there is a nontrivial left-orderable Property~T group), asked in \cite[(7.8)]{BHV});
\item\label{t_cir} there is an infinite Property~T subgroup in $\mathrm{Homeo}(\R/\Z)$;
\item\label{t_pw} there is an infinite Property~T subgroup in $\PC(\R/\Z)$.
\end{enumerate}
\end{prop}
\begin{proof}
Since there are inclusions $\mathrm{Homeo}(\R)\to\mathrm{Homeo}(\R/\Z)\to\PC(\R/\Z)$, the implications (\ref{t_hor})$\Rightarrow$(\ref{t_cir})$\Rightarrow$(\ref{t_pw}) are obvious.

Suppose (\ref{t_cir}): $\Gamma\subset\mathrm{Homeo}^+(\R/\Z)$ is infinite with Property~T. Let $\tilde{\Gamma}$ be its inverse image in $\mathrm{Homeo}(\R)$.  If $\tilde{\Gamma}$ has Property~T, we are done. Otherwise, by \cite[Theorem 1.7.11]{BHV}, $\tilde{\Gamma}$ has infinite abelianization. Hence its derived subgroup $\tilde{\Gamma}'$ embeds as a finite index subgroup of $\Gamma$ and hence has Property~T, proving (\ref{t_hor}). 

Suppose (\ref{t_pw}). Arguing as in the proof of Theorem \ref{main} (with no geometric structure beyond being a 1-dimensional topological manifold without boundary), we obtain $\Gamma$ infinite with Property~T in the homeomorphism group of a finitely-charted Hausdorff 1-manifold. Passing to a finite index subgroup, it preserves all components, and the action on some component yields either (\ref{t_hor}) or (\ref{t_cir}).
\end{proof}

\end{document}